\documentclass[12pt]{amsart}
\usepackage{amssymb}
\usepackage{enumerate}

\makeatletter
\@namedef{subjclassname@2020}{%
  \textup{2020} Mathematics Subject Classification}
\makeatother

\newtheorem{theorem}{Theorem}[section]
\newtheorem{corollary}[theorem]{Corollary}
\newtheorem{lemma}[theorem]{Lemma}
\newtheorem{proposition}[theorem]{Proposition}

\newtheorem{conjecture}[equation]{Conjecture}

\theoremstyle{definition}
\newtheorem{definition}[theorem]{Definition}
\newtheorem{remark}[theorem]{Remark}

\newcommand{\ud}[0]{\,\mathrm{d}}

\newcommand{\abs}[1]{|#1|}

\newcommand{\Norm}[2]{\|#1\|_{#2}}

\newcommand{\ave}[1]{\langle #1\rangle}

\newcommand{\BMO}[0]{\operatorname{BMO}}
\newcommand{\CMO}[0]{\operatorname{CMO}}

\newcommand{\supp}[0]{\operatorname{supp}}
\newcommand{\loc}[0]{\operatorname{loc}}

\newcommand{\testi}{{\mathcal S}}

\newcommand{\R}{\mathbb{R}}
\newcommand{\C}{\mathbb{C}}

\newcommand{\eps}[0]{\varepsilon}

\begin{document}

\baselineskip=17pt

\title[Extrapolation of compactness]{Extrapolation of compactness on\\ weighted Morrey spaces}

\author[S. Lappas]{Stefanos Lappas}
\address{Department of Mathematics and Statistics\\
P.O.Box~68 (Pietari Kalmin katu~5) FI-00014\\ 
University of Helsinki\\ Finland}
\email{stefanos.lappas@helsinki.fi}

\date{}

\begin{abstract}
In a previous work, ``compact versions'' of Rubio de Francia's weighted extrapolation theorem were proved, which allow one to extrapolate the compactness of an linear operator from just one space to the full range of weighted Lebesgue spaces, where this operator is bounded. In this paper, we extend these results to the setting of weighted Morrey spaces. As applications, we easily obtain new results on the weighted compactness of commutators of Calder\'on--Zygmund singular integrals, rough singular integrals and Bochner--Riesz multipliers.
\end{abstract}



\maketitle

\section{Introduction}

We refer to a locally integrable positive almost everywhere function $w$ on $\R^d$ as a weight and we define the weighted Lebesgue and Morrey spaces as follows: 

\begin{definition}
Let $1\leq p<\infty$ and $w$ be a weight. Then a weighted Lebesgue space is defined by
\begin{equation*}
  L^p(w):=\Big\{f:\R^d\to\C\text{ measurable }\Big|\ \Norm{f}{L^p(w)}:=\Big(\int_{\R^d}\abs{f}^p w\Big)^{1/p}<\infty\Big\}.
\end{equation*}
\end{definition}

\begin{definition}[\cite{Samko}]\label{weighted Morrey space}
Let $1\leq p<\infty$, $0<\lambda<d$ and $w$ be a weight. Then the Samko type weighted Morrey space is defined by
\begin{equation*}
  \mathcal L^{p,\lambda}(w):=\Big\{f\in L^p_{loc}(w)\Big|\ \Norm{f}{\mathcal L^{p,\lambda}(w)}:=\sup_Q |Q|^{-\frac{\lambda}{dp}}\Big(\int_{Q}\abs{f}^p w\Big)^{1/p}<\infty\Big\},
\end{equation*}
where the supremum is taken over all cubes $Q\subset\R^d$.
\end{definition}

\begin{remark}
Alternatively, we could define the Samko type weighted Morrey spaces with balls instead of cubes. If $w\equiv 1$, then $\mathcal L^{p,\lambda}(w)=\mathcal L^{p,\lambda}(\R^d)$, where $\mathcal L^{p,\lambda}(\R^d)$ is the classical Morrey space (see \cite{Morrey}).
\end{remark}

As we will work with Muckenhoupt weight characteristics, we recall the following definitions: 

\begin{definition}[\cite{Gehring, Muckenhoupt:Ap}]\label{def:Muchenhoupt weights}
A weight $w\in L_{\loc}^1(\R^d)$ is called a Muckenhoupt $A_p(\R^d)$ weight (or $w\in A_p(\R^d)$) if 
\begin{equation*}
\begin{split}
  &[w]_{A_p}:=\sup_Q\ave{w}_Q\ave{w^{-\frac{1}{p-1}}}_Q^{p-1}<\infty,\qquad 1<p<\infty, \\
  &[w]_{A_1}:=\sup_Q\ave{w}_Q\Norm{w^{-1}}{L^\infty(Q)}<\infty,\qquad p=1,
\end{split}
\end{equation*}
where the supremum is taken over all cubes $Q\subset\R^d$ and $\ave{w}_Q:=\abs{Q}^{-1}\int_Q w$. 

A weight $w$ is said to belong to the reverse H\"older class $RH_{\sigma}(\R^d)$ (or $w\in RH_{\sigma}(\R^d)$) if
\begin{equation*}
\begin{split}
  &[w]_{RH_\sigma}:=\sup_Q\ave{w^{\sigma}}_Q^{1/\sigma}\ave{w}_Q^{-1}<\infty,\qquad 1<\sigma<\infty,  \\
  &[w]_{RH_\infty}:=\sup_Q\Norm{w}{L^\infty(Q)}\ave{w}_Q^{-1}<\infty,\qquad \sigma=\infty.
\end{split}
\end{equation*}
\end{definition}

The classes $RH_{\sigma}(\R^d)$ and $A_p(\R^d)$ were introduced to study the {\em $L^p$-integra\-bility} of the partial derivatives of a quasiconformal mapping and the weighted norm inequalities for {\em Hardy--Littlewood maximal function}, respectively; see \cite{Gehring, Muckenhoupt:Ap}.

The following theorem of Rubio de Francia \cite{Rubio:factorAp} (see also the work of Auscher--Martell \cite{AM}) on the extrapolation of {\em boundedness} on weighted spaces is one of the most important tools of modern harmonic analysis:

\begin{theorem}[\cite{AM}, Theorem 4.9 and \cite{Rubio:factorAp}]\label{thm:limited range extrp.}
Let $1\leq p_{-}<p_{+}\leq\infty$, and $T$ be a linear operator simultaneously defined and bounded on $L^{p_1}(\tilde w)$ for {\bf some} $p_1\in[p_{-},p_{+}]$ and {\bf all} $\tilde w\in A_{p_1/p_{-}}(\R^d)\cap RH_{(p_{+}/p_1)'}(\R^d)$. Then $T$ is also defined and bounded on $L^p(w)$ for all $p\in(p_{-},p_{+})$ and all $w\in A_{p/p_{-}}(\R^d)\cap RH_{(p_{+}/p)'}(\R^d)$.
\end{theorem}

\begin{remark}
We don't seem to have an analogue of Rubio de Francia's extrapolation theorem on weighted Morrey spaces.
\end{remark}

In \cite[Theorem 3.31]{CUMP:book}, a version of Theorem \ref{thm:limited range extrp.} is stated in terms of non-negative measurable pairs of functions $(f,g)$. This means that one does not need to work with specific operators since nothing about the operators themselves is used (like linearity or sublinearity) and they play no role. However, we work with the pairs $(|f|,|Tf|)$, where $T$ is a linear operator, since the abstract compactness results that we will use in order to prove Theorem \ref{thm:main} below hold for linear operators (see Theorem \ref{thm:CwKa} of Cwikel--Kalton and Theorem \ref{thm:CR} of Cwikel--Rochberg).

In the recent paper \cite{HL}, the authors provided the following version
for extrapolation of {\em compactness}: (See also \cite{COY,HL2020} for extensions to multilinear operators.)

\begin{theorem}[\cite{HL}, Theorems 1.3 and 2.4]\label{thm:limited range extrp.compact}
In the setting of Theorem \ref{thm:limited range extrp.}, suppose in addition that $T$ is compact on $L^{p_1}(w_1)$ for {\bf some} $w_1\in A_{p_1/p_{-}}(\R^d)\cap RH_{(p_{+}/p_1)'}(\R^d)$. Then $T$ is compact on $L^{p}(w)$ for all $p\in(p_{-},p_{+})$ and all $w\in A_{p/p_{-}}(\R^d)\cap RH_{(p_{+}/p)'}(\R^d)$.
\end{theorem}

In this paper, we extend Theorem \ref{thm:limited range extrp.compact} to the setting of weighted Morrey space. In particular, we obtain the following:

\begin{theorem}\label{thm:main}
Let $0<\lambda<d$, $1\leq p_{-}<p_{+}\leq\infty$ and $T$ be a linear operator simultaneously defined and bounded on $\mathcal L^{p,\lambda}(w)$ for {\bf all} $p\in(p_{-},p_{+})$ and {\bf all} $w\in A_{p/p_{-}}(\R^d)\cap RH_{(p_{+}/p)'}(\R^d)$. Suppose in addition that $T$ is compact on $\mathcal L^{p_1,\lambda}(w_1)$ for {\bf some}  $p_1\in[p_{-},p_{+}]$ and {\bf some} $w_1\in A_{p_1/p_{-}}(\R^d)\cap RH_{(p_{+}/p_1)'}(\R^d)$. Then $T$ is compact on $\mathcal L^{p,\lambda}(w)$ for all $p\in(p_{-},p_{+})$ and all $w\in A_{p/p_{-}}(\R^d)\cap RH_{(p_{+}/p)'}(\R^d)$. 
\end{theorem}

\begin{remark}
Theorem \ref{thm:main} remains true in the case $p_{+}=\infty$, provided that $p_1<\infty$. Thus the reverse H\"older conditions on $w,w_1$ are vacuous. Due to Theorem \ref{thm:HNS} below, it seems that our results don't apply to other type of weighted Morrey spaces such as the Komori--Shirai type weighted Morrey space considered in \cite{KS}. In addition, notice that if we let $\lambda\to 0$ in Definition \ref{weighted Morrey space}, then $\mathcal {L}^{p,0}(w)\equiv L^p(w)$ and hence Theorem \ref{thm:main} formally recovers Theorem \ref{thm:limited range extrp.compact}.
\end{remark}

When $w_1\equiv1$, Theorem \ref{thm:main} says that we can obtain weighted compactness if the weighted boundedness and unweighted compactness are already known. This case is relevant to all our applications in Sections \ref{sec:4}--\ref{sec:7}.

The paper is organized as follows: in Section \ref{sec:main results} we collect some previously known results from which the proof of Theorem \ref{thm:main} will follow in Section \ref{sec:3}. In Sections \ref{sec:4}--\ref{sec:7} we provide several applications of our main result. An example of these applications to commutators of {\em Calder\'on--Zygmund operators} is the following (we refer to Sections \ref{sec:4} and \ref{sec:5} for the notions of $\CMO(\R^d)$ and Calder\'on--Zygmund operators, respectively):

\begin{theorem}\label{thm:app2}
Let $b\in\CMO(\R^d)$, $T$ be a Calder\'on--Zygmund operator that extends boundedly to $L^2(\R^d)$ and satisfies for all $f\in C_c^\infty(\R^d)$ the condition $Tf(x)=\lim_{\eps\to 0}\int_{|x-y|\ge\eps}K(x,y)f(y)dy$ a.e. $x\in\R^d$. Then the commutator $[b,T]$ is compact on $\mathcal{L}^{p,\lambda}(w)$ for all $w\in A_{s}(\R^d)\cap RH_{t}(\R^d)$ and all $p,\lambda,s,t$ such that 
\begin{equation*}
  p\in(1,\infty),\quad 0<\lambda<d, \quad s\in\bigg[1,\min\bigg\{p,\frac{d}{\lambda}\bigg\}\bigg], \quad t\in\bigg(\bigg(\frac{d}{s\lambda}\bigg)',\infty\bigg).  
\end{equation*}
\end{theorem}

See also Theorems \ref{thm:app3} and \ref{thm:app4} for similar results on {\em rough singular integrals} and {\em Bochner--Riesz multipliers}. Although compactness of such operators on the unweighted Morrey spaces has been considered in the literature, obtaining compactness results on weighted Morrey spaces appears to be entirely new altogether.

\subsection*{Notation} Throughout the paper, we denote by $C$ a positive constant which is independent of the main parameters but it may change at each occurrence, and we write $f\lesssim g$ if $f\leq Cg$. The term cube will always refer to a cube $Q\subset\R^d$ and $|Q|$ will denote its Lebesgue measure. Recall from Definition \ref{def:Muchenhoupt weights} that $\ave{w}_Q$ denotes $\abs{Q}^{-1}\int_Q w$, the average of $w$ over $Q$, and $p'$ is the conjugate exponent to $p$, that is $p':=p/(p-1)$.

\section{Preliminaries}\label{sec:main results}

We collect the results from which the proof of Theorem \ref{thm:main} will follow in Section \ref{sec:3}. 

Let $S:=\{z\in\C:0<Re(z)<1\}$ and $\bar{S}:=\{z\in\C:0\leq Re(z)\leq 1\}$. 
Following \cite{Calderon, HNS}, we recall the following definitions of two complex interpolation functors: 

\begin{definition}[Calder\'on's first complex interpolation space] Let $(X_0,X_1)$ be a compatible couple of Banach spaces.
\begin{enumerate}[1.]
  \item Define $\mathcal{F}(X_0,X_1)$ as the set of all functions $F:\bar{S}\to X_0+X_1$ such that
  \begin{enumerate}[(a)]
  \item $F$ is continuous on $\bar{S}$ and $\sup_{z\in\bar{S}}\|F(z)\|_{X_0+X_1}<\infty$,
  \item $F$ is holomorphic on $S$,
  \item the functions $t\in\R\mapsto F(j+it)\in X_j$ are bounded and continous on $\R$ for $j=0,1$.
  \end{enumerate}
The space $\mathcal{F}(X_0,X_1)$ is equipped with the norm
\begin{equation*}
  \|F\|_{\mathcal{F}(X_0,X_1)}:=\max\bigg\{\sup_{t\in\R}\|F(it)\|_{X_0},\;\;\sup_{t\in\R}\|F(1+it)\|_{X_1}\bigg\}.
\end{equation*}
  \item Let $\theta\in(0,1)$. Define the complex interpolation space $[X_0,X_1]_{\theta}$ with respect to $(X_0,X_1)$ to be the set of all functions $x\in X_0+X_1$ such that $x=F(\theta)$ for some $F\in\mathcal{F}(X_0,X_1)$. The norm on $[X_0,X_1]_{\theta}$ is defined by 
  \begin{equation*}
  \|x\|_{[X_0,X_1]_{\theta}}:=\inf\{\|F\|_{\mathcal F(X_0, X_1)}: x=F(\theta)\;\;\text{for some}\;\; F\in\mathcal{F}(X_0,X_1)\}.
  \end{equation*}
\end{enumerate}
\end{definition}

Let $Y$ be a Banach space. We let 
\begin{equation*}
  \text{Lip}(\R,Y):=\Big\{f:\R\to Y\text{ continuous }\Big|\ \Norm{f}{ \text{Lip}(\R,Y)}<\infty\Big\}.
\end{equation*}
where
\begin{equation*}
  \text{Lip}(\R,Y):=\sup_{-\infty<s<t<\infty}\frac{\|f(t)-f(s)\|_Y}{|t-s|}.
\end{equation*}

\begin{definition}[Calder\'on's second complex interpolation space] Suppose that $\bar{X}=(X_0,X_1)$ is a compatible couple of Banach spaces.
\begin{enumerate}[1.]
  \item Define $\mathcal{G}(X_0,X_1)$ as the set of all functions $F:\bar{S}:\to X_0+X_1$ such that
  \begin{enumerate}[(a)]
  \item $F$ is continuous on $\bar{S}$ and $\sup_{z\in \bar{S}}\big\|\frac{F(z)}{1+|z|}\big\|_{X_0+X_1}<\infty$,
  \item $F$ is holomorphic on $S$,
  \item the functions $t\in\R\mapsto F(j+it)\in X_j$ are Lipschitz continuous on $\R$ for $j=0,1$.
  \end{enumerate}
The space $\mathcal{G}(X_0,X_1)$ is equipped with the norm
\begin{equation*}
  \|F\|_{\mathcal{G}(X_0,X_1)}:=\max\{\|F(i\cdot)\|_{\text{Lip}(\R,X_0)},\;\;\|F(1+i\cdot)\|_{\text{Lip}(\R,X_1)}\}.
\end{equation*}
  \item Let $\theta\in(0,1)$. Define the complex interpolation space $[X_0,X_1]^{\theta}$ with respect to $(X_0,X_1)$ to be the set of all functions $x\in X_0+X_1$ such that $x=F'(\theta)$ for some $F\in\mathcal{G}(X_0,X_1)$. The norm on $[X_0,X_1]^{\theta}$ is defined by 
  \begin{equation*}
  \|x\|_{[X_0,X_1]^{\theta}}:=\inf\{\|F\|_{\mathcal G(X_0, X_1)}: x=F'(\theta)\;\;\text{for some}\;\; F\in\mathcal{G}(X_0,X_1)\}.
  \end{equation*}
\end{enumerate}
\end{definition}

Our main abstract tools are the following theorems of Cwikel--Kalton \cite{CwKa} and Cwikel--Rochberg \cite{CR}:

\begin{theorem}[\cite{CwKa}]\label{thm:CwKa}
Let $(X_0,X_1)$ and $(Y_0,Y_1)$ be Banach couples and $T$ be a linear operator such that
$T:X_0+X_1\to Y_0+Y_1$ and $T:X_j\to Y_j$ boundedly for $j=0,1$.
Suppose moreover that $T:X_1\to Y_1$ is compact.
Then also $T:[X_0,X_1]_\theta\to[Y_0,Y_1]_\theta$ is compact for $\theta\in(0,1)$ under {\bf any} of the following four side conditions:
\
\begin{enumerate}
  \item\label{it:UMD} $X_1$ has the UMD (unconditional martingale differences) property,
  \item\label{it:Xinterm} $X_1$ is reflexive, and $X_1=[X_0,E]_\alpha$ for some Banach space $E$ and $\alpha\in(0,1)$,
  \item\label{it:Yinterm} $Y_1=[Y_0,F]_\beta$ for some Banach space $F$ and $\beta\in(0,1)$,
  \item\label{it:lattice} $X_0$ and $X_1$ are both complexified Banach lattices of measurable functions on a common measure space.
\end{enumerate}
\end{theorem}

We have swapped the roles of the indices $0$ and $1$ in comparison to \cite{CwKa}. For the UMD property, see \cite[Ch. 4]{HNVW1}.

\begin{theorem}[\cite{CR}, Theorem 2.2]\label{thm:CR}
Let $(X_0,X_1)$ and $(Y_0,Y_1)$ be arbitrary Banach couples. Let $T$ be a bounded linear operator such that $T:X_0+X_1\to Y_0+Y_1$. Suppose moreover that $T:[X_0,X_1]_{\theta}\to[Y_0,Y_1]_{\theta}$ is compact for some $\theta\in(0,1)$. Then also $T:[X_0,X_1]^{\theta}\to[Y_0,Y_1]_{\theta}$ is compact for the same $\theta\in(0,1)$. 
\end{theorem}

We consider a measurable function $f$, weights $v,v_0,v_1$ and numbers $p_0,p,\\p_1$, $R>0$ fixed. The following objects depend on these quantities, but we do not indicate this in the notation. We define
\begin{equation*}
\begin{split}
 E_{R,0}:&=\bigg\{x\in\R^d:|f(x)|^{p_0-p}\frac{v_0(x)}{v(x)}\ge R\bigg\},  \\  
 E_{R,1}:&=\bigg\{x\in\R^d:|f(x)|^{p_1-p}\frac{v_1(x)}{v(x)}\ge R\bigg\},  \\ 
 E_R:&=E_{R,0}\cup E_{R,1}.
\end{split}
\end{equation*}
Define
\begin{equation*}
  f_R:=f(1-\chi_{E_R})
\end{equation*}
for $f\in\mathcal{L}^{q,\lambda}(v)$ and consider the following condition:
\begin{equation}\label{main cond.}
f=\lim_{R\to\infty}f_R\;\;\text{in}\;\;\mathcal{L}^{q,\lambda}(v).
\end{equation}

We will use Theorems \ref{thm:CwKa} and \ref{thm:CR} in the following special setting: 

\begin{proposition}\label{prop:main}
Let $0<\lambda<d$, let $1\leq p_{-}<p_{+}\leq\infty$ and  $q_1\in[p_{-},p_{+}]$, $q\in(p_{-},p_{+})$
\begin{equation*}
v\in A_{q/p_{-}}(\R^d)\cap RH_{(p_{+}/q)'}(\R^d),\qquad v_1\in A_{q_1/p_{-}}(\R^d)\cap RH_{(p_{+}/q_1)'}(\R^d). 
\end{equation*}
Then
\begin{equation*}
\begin{split}
  [\mathcal L^{q_0,\lambda}(v_0),\mathcal L^{q_1,\lambda}(v_1)]^{\gamma}&=\mathcal L^{q,\lambda}(v)  \\
  [\mathcal L^{q_0,\lambda}(v_0),\mathcal L^{q_1,\lambda}(v_1)]_{\gamma}&=\{f\in \mathcal L^{q,\lambda}(v):
  \eqref{main cond.}\quad\text{holds}\}
\end{split}
\end{equation*}
for some $q_0\in(p_{-},p_{+})$,  $v_0\in A_{q_0/p_{-}}(\R^d)\cap RH_{(p_{+}/q_0)'}(\R^d)$, and $\gamma\in(0,1)$.
\end{proposition}

We postpone the proof of Proposition \ref{prop:main} to the following section.

\begin{lemma}\label{lem:lemOk}
If $0<\lambda<d$, $1\leq p_j<\infty$ and $w_j$ are weights, then the spaces $X_j=\mathcal L^{p_j,\lambda}(w_j)$ satisfy the condition \eqref{it:lattice} of Theorem \ref{thm:CwKa}.
\end{lemma}

\begin{proof}
It is easy to verify that $X_j=\mathcal L^{p_j,\lambda}(w_j)$ are complexified Banach lattices of measurable functions on the common measure space $\R^d$ (see also \cite{ST}).
\end{proof}

\begin{remark}
We observe that Morrey spaces don't satisfy any of the conditions \eqref{it:UMD}, \eqref{it:Xinterm} and \eqref{it:Yinterm} of Theorem \ref{thm:CwKa}. For more details, see \cite[Theorem 4.3.3]{HNVW1} and \cite{ST}.
\end{remark}

We quote the following results from which the proof of Proposition \ref{prop:main} will follow:

\begin{theorem}[\cite{HNS}, Theorem 2.3]\label{thm:HNS}
If $0<\lambda<d$, $q_0,q_1\in[1,\infty)$ and $w_0,w_1$ are two weights, then for all $\theta\in(0,1)$ we have
\begin{equation*}
\begin{split}
  [\mathcal L^{q_0,\lambda}(w_0),\mathcal L^{q_1,\lambda}(w_1)]^{\theta}&=\mathcal L^{q,\lambda}(w)  \\
  [\mathcal L^{q_0,\lambda}(w_0),\mathcal L^{q_1,\lambda}(w_1)]_{\theta}&=\{f\in \mathcal L^{q,\lambda}(w):
  \eqref{main cond.}\quad\text{holds}\},
\end{split}
\end{equation*}
where
\begin{equation}\label{eq:convexity}
  \frac{1}{q}:=\frac{1-\theta}{q_0}+\frac{\theta}{q_1},\qquad
  w^{\frac{1}{q}}:=w_0^{\frac{1-\theta}{q_0}}w_1^{\frac{\theta}{q_1}}.
\end{equation}
\end{theorem}

\begin{remark}
The authors of \cite{HNS} proved Theorem \ref{thm:HNS} in the framework of generalized weighted Morrey spaces. As explained in \cite[Example 1.2]{HNS}, if one takes $\varphi(x,r)=|B(x,r)|^{\frac{1}{q}-\frac{\lambda}{dq}}$and $v=1$ in \cite[Definition 1.1]{HNS}, where  $B=B(x,r)$ is a ball with center $x$ and radius $r$ and $v$ is a weight, then the Samko type weighted Morrey space is an example of the generalized weighted Morrey space.
\end{remark}

In order to connect Theorem \ref{thm:HNS} with the $A_{q/p_{-}}(\R^d)\cap RH_{(p_{+}/q)'}(\R^d)$ weights, we need:

\begin{lemma}[\cite{HL}, Lemma 4.9]\label{lem:main}
Let $1\leq p_{-}<p_{+}\leq\infty$, $q_1\in[p_{-},p_{+}]$, $q\in(p_{-},p_{+})$, and
\begin{equation*}
  w_1\in A_{q_1/p_{-}}(\R^d)\cap RH_{(p_{+}/q_1)'}(\R^d),\qquad w\in A_{q/p_{-}}(\R^d)\cap
  RH_{(p_{+}/q)'}(\R^d).
\end{equation*}
Then there exist $q_0\in(p_{-},p_{+})$, $w_0\in A_{q_0/p_{-}}(\R^d)\cap RH_{(p_{+}/q_0)'}(\R^d)$, and $\theta\in(0,1)$ such that (\ref{eq:convexity}) holds.
\end{lemma}

\begin{remark}
Lemma \ref{lem:main} remains true in the case $p_{+}=\infty$, provided that $q_1<\infty$. In this case the reverse H\"older conditions on $w,w_0,w_1$ are vacuous and the proof is given in \cite[Lemma 4.4]{HL}.
\end{remark}

\section{The Proof of the key Proposition \ref{prop:main} and Theorem \ref{thm:main}}\label{sec:3}

To complete the proof of Theorem \ref{thm:main} it remains to verify Proposition \ref{prop:main}: 

\begin{proof}[Proof of Proposition \ref{prop:main}]
We prove the proposition in the case $p_{+}<\infty$. The case $p_{+}=\infty$ is proved in a similar way. With some $0<\lambda<d$, we are given $1\leq p_{-}<p_{+}<\infty$, $q_1\in[p_{-},p_{+}]$, $q\in(p_{-},p_{+})$ and weights $v\in A_{q/p_{-}}(\R^d)\cap RH_{(p_{+}/q)'}(\R^d)$, $v_1\in A_{q_1/p_{-}}(\R^d)\cap RH_{(p_{+}/q_1)'}(\R^d)$. By Lemma \ref{lem:main}, there is some  $q_0\in(p_{-},p_{+})$, a weight $v_0\in A_{q_0/p_{-}}(\R^d)\cap RH_{(p_{+}/q_0)'}(\R^d)$, and $\theta\in(0,1)$ such that
\begin{equation*}
  \frac{1}{q}=\frac{1-\theta}{q_0}+\frac{\theta}{q_1},
  \qquad
  v^{\frac{1}{q}}=v_0^{\frac{1-\theta}{q_0}}v_1^{\frac{\theta}{q_1}}.
\end{equation*}
By Theorem \ref{thm:HNS}, we then have \begin{equation*}
\begin{split}
  [\mathcal L^{q_0,\lambda}(v_0),\mathcal L^{q_1,\lambda}(v_1)]^{\theta}&=\mathcal L^{q,\lambda}(v)  \\
  [\mathcal L^{q_0,\lambda}(v_0),\mathcal L^{q_1,\lambda}(v_1)]_{\theta}&=\{f\in \mathcal L^{q,\lambda}(v):
  \eqref{main cond.}\quad\text{holds}\},
\end{split}
\end{equation*}
as we claimed.
\end{proof}

We can now give the proof of our main result:

\begin{proof}[Proof of Theorem \ref{thm:main}]
Let $p_{+}<\infty$ and recall that the assumptions of Theorem \ref{thm:main} are in force. The case $p_{+}=\infty$ is proved in a similar way. In particular, $T$ is a bounded linear operator on $\mathcal L^{p,\lambda}(w)$ for all $p\in(p_{-},p_{+})$ and all $w\in A_{p/p_{-}}(\R^d)\cap RH_{(p_{+}/p)'}(\R^d)$. In addition, it is assumed that $T$ is a compact operator on $\mathcal L^{p_1,\lambda}(w_1)$ for some $p_1\in[p_{-},p_{+}]$ and some $w_1\in A_{p_1/p_{-}}(\R^d)\cap RH_{(p_{+}/p_1)'}(\R^d)$. We need to prove that $T$ is actually compact on $\mathcal L^{p,\lambda}(w)$ for all $p\in(p_{-},p_{+})$ and all $w\in A_{p/p_{-}}(\R^d)\cap RH_{(p_{+}/p)'}(\R^d)$. Now, fix some $p\in(p_{-},p_{+})$ and $w\in A_{p/p_{-}}(\R^d)\cap RH_{(p_{+}/p)'}(\R^d)$. By Proposition \ref{prop:main}, we have
\begin{equation*}
\begin{split}
  [\mathcal L^{p_0,\lambda}(w_0),\mathcal L^{p_1,\lambda}(w_1)]^{\theta}&=\mathcal L^{p,\lambda}(w)  \\
  [\mathcal L^{p_0,\lambda}(w_0),\mathcal L^{p_1,\lambda}(w_1)]_{\theta}&=\{f\in \mathcal L^{p,\lambda}(w):
  \eqref{main cond.}\quad\text{holds}\}
\end{split}
\end{equation*}
for some $p_0\in(p_{-},p_{+})$, some $w_0\in A_{p_0/p_{-}}(\R^d)\cap RH_{(p_{+}/p_0)'}(\R^d)$ and some $\theta\in(0,1)$. Writing $X_j=Y_j=\mathcal L^{p_j,\lambda}(w_j)$, we know that $T:X_0+X_1\to Y_0+Y_1$, that $T:X_0\to Y_0$ is bounded (since $T$ is bounded on all $\mathcal L^{q,\lambda}(w)$ with $q\in(p_{-},p_{+})$ and $w\in A_{q/p_{-}}(\R^d)\cap RH_{(p_{+}/q)'}(\R^d)$), and that $T:X_1\to Y_1$ is compact (since this was assumed). By Lemma \ref{lem:lemOk}, the last condition \eqref{it:lattice} of Theorem \ref{thm:CwKa} is also satisfied by these spaces $X_j=\mathcal L^{p_j,\lambda}(w_j)$. By Theorem \ref{thm:CwKa}, it follows that $T$ is also compact on $[X_0,X_1]_\theta=[Y_0,Y_1]_\theta=\{f\in \mathcal L^{p,\lambda}(w):\eqref{main cond.}\quad\text{holds}\}$. Hence, by Theorem \ref{thm:CR}, we conclude that $T$ is also compact from  $[X_0,X_1]^\theta=\mathcal L^{p,\lambda}(w)$ to $[Y_0,Y_1]_\theta=\{f\in \mathcal L^{p,\lambda}(w):\eqref{main cond.}\quad\text{holds}\}$. In particular, this implies that $T$ is also compact on $\mathcal L^{p,\lambda}(w)$.
\end{proof}

\section{Commutators with functions of bounded mean oscillation}\label{sec:4}

We indicate several applications of Theorem \ref{thm:main} which deal with commutators of the form
\begin{equation*}
  [b,T]:f\mapsto bT(f)-T(bf),
\end{equation*}
where the pointwise multiplier $b$ belongs to the space
\begin{equation*}
  \BMO(\R^d):=\Big\{f:\R^d\to\C\ \Big|\ \Norm{f}{\BMO}:=\sup_Q\ave{\abs{f-\ave{f}_Q}}_Q<\infty\Big\}
\end{equation*}
 of functions of bounded mean oscillation, or its subspace
 \begin{equation*}
  \CMO(\R^d):=\overline{C_c^\infty(\R^d)}^{\BMO(\R^d)},
\end{equation*}
where the closure is in the $\BMO$ norm. We will need the following results of Duoandikoetxea--Rosenthal \cite{DR} (see also \cite{DR2018, DR2020}) on the extrapolation of boundedness on the corresponding weighted Morrey spaces from the assumption of weighted estimates on $L^p(w)$ spaces:

\begin{theorem}[\cite{DR}, Theorem 1.1]\label{thm:extr. bdd. 1}
Let $1\leq\kappa\leq p_1<\infty$ and $T$ be a defined and bounded operator on $L^{p_1}(w)$ for all $w\in A_{p_1/\kappa}(\R^d)$. Then $T$ is bounded on $\mathcal L^{p,\lambda}(w)$ for all $p\in(\kappa,\infty)$ (and also $p=\kappa$ if $p_1=\kappa$), all $0<\lambda<d-\frac{d}{\sigma}$ and all $w\in A_{p/\kappa}(\R^d)\cap RH_{\sigma}(\R^d)$.
\end{theorem}

\begin{theorem}[\cite{DR}, Corollary 3.2]\label{thm:extr. bdd. 2}
Let $1\leq p_{-}\leq p_1\leq p_{+}<\infty$ and $T$ be a defined and bounded operator on $L^{p_1}(w)$ for all $w\in A_{p_1/p_{-}}(\R^d)\cap RH_{(p_{+}/p_1)'}(\R^d)$. Then for all $p\in(p_{-},p_{+})$ (and also $p=p_{-}$ if $p_1=p_{-}$) and $\sigma>(p_{+}/p)'$, $T$ is bounded on $\mathcal L^{p,\lambda}(w)$ for all $0<\lambda<d\bigg(1-\frac{p}{p_{+}}-\frac{1}{\sigma}\bigg)$ and all $w\in A_{p/p_{-}}(\R^d)\cap RH_{\sigma}(\R^d)$.
\end{theorem}

\begin{remark}
In \cite{DR}, Theorems \ref{thm:extr. bdd. 1} and \ref{thm:extr. bdd. 2} are stated in terms of non-negative measurable pairs of functions $(f,g)$. This has the advantage of providing immediately several different versions. In our applications below, we will apply these with $[b,T]$ in place of $T$. 
\end{remark}

In \cite{ABKP,BMMST}, the following general results on weighted boundedness about commutators were obtained:

\begin{theorem}[\cite{ABKP}]\label{thm:commuBasic}
Let $1\leq\kappa<p_1<\infty$, and $T$ be a linear operator defined and bounded on $L^{p_1}(\tilde w)$ for all $\tilde w\in A_{p_1/\kappa}(\R^d)$, with the operator norm dominated by some increasing function of $[\tilde w]_{A_{p_1/\kappa}}$. Suppose also that $b\in\BMO(\R^d)$. Then also $[b,T]$ extends to a bounded linear operator on $L^{p_1}(\tilde w)$ for all $\tilde w\in A_{p_1/\kappa}(\R^d)$, and its operator norm is dominated by another increasing function of $[\tilde w]_{A_{p_1/\kappa}}$.
\end{theorem}

The statement in \cite[Theorem 2.13]{ABKP} is somewhat more general, but the above particular case is easily seen to be contained in it.

\begin{theorem}[\cite{BMMST}, Corollary 5.3]\label{coro:restricted}
Let $1\leq p_{-}<p_1<p_{+}\leq\infty$, and $T$ be a linear operator bounded on $L^{p_1}(\tilde w)$ for all $\tilde w\in A_{\frac{p_1}{p_{-}}}(\R^d)\cap RH_{\big(\frac{p_{+}}{p_1}\big)'}(\R^d)$. If $b\in\BMO(\R^d)$, then $[b,T]$ is bounded on $L^{p_1}(\tilde w)$ for all $\tilde w\in A_{\frac{p_1}{p_{-}}}(\R^d)\cap RH_{\big(\frac{p_{+}}{p_1}\big)'}(\R^d)$.
\end{theorem}

A combination of Theorems \ref{thm:main}, \ref{thm:extr. bdd. 1}, \ref{thm:extr. bdd. 2}, \ref{thm:commuBasic} and \ref{coro:restricted} immediately gives the following two corollaries:
 
\begin{corollary}\label{cor:commuBasic 1}

Let $1\leq\kappa<p_{+}<\infty$, and $1<\sigma_1<\infty$, and 
\begin{equation*}
  p_{-}=\max\bigg\{\kappa,p_{+}\bigg(1-\frac{1}{\sigma_1}\bigg)\bigg\},
\end{equation*}
and $0<\lambda_1<d-\frac{d}{\sigma_1}$. Suppose moreover that:
\begin{enumerate}
  \item $T$ is a linear operator defined and bounded on $L^{\tilde p_1}(\tilde w)$ for some $\tilde p_1\in(\kappa,\infty)$ and for all $\tilde w\in A_{\tilde p_1/\kappa}(\R^d)$, with the operator norm dominated by some increasing function of $[\tilde w]_{A_{\tilde p_1/\kappa}}$,
  \item the commutator $[b,T]$ is compact on $\mathcal L^{p_1,\lambda_1}(w_1)$ for some 
  \begin{equation*}
  b\in\BMO(\R^d),
  \end{equation*}
  some $p_1\in[p_{-},p_{+}]$ and some $w_1\in A_{p_1/p_{-}}(\R^d)\cap RH_{(p_{+}/p_1)'}(\R^d)$.
\end{enumerate}
Then $[b,T]$ is compact on $\mathcal L^{p,\lambda_1}(w)$  for the same $b$, for all $p\in(p_{-},p_{+})$ and all $w\in A_{p/p_{-}}(\R^d)\cap RH_{(p_{+}/p)'}(\R^d)$.
\end{corollary}

\begin{proof}
We verify the assumptions of Theorem \ref{thm:main} for the fixed numbers $\lambda_1,p_1,\sigma_1,\kappa,p_{+},d$ and the weight $w_1$ appearing in the statement of the Corollary, and the operator $[b,T]$ in place of $T$. By Theorem \ref{thm:commuBasic}, $[b,T]$ is bounded on $L^{\tilde p_1}(\tilde w)$ for the same $\tilde p_1\in(\kappa,\infty)$ and all $\tilde w\in A_{p_1/\kappa}(\R^d)$. Then, by Theorem \ref{thm:extr. bdd. 1} with $[b,T]$ in place of $T$, $[b,T]$ is bounded on $\mathcal L^{p,\lambda}(\tilde w)$ for all $p\in(\kappa,\infty)$, all $0<\lambda<d-\frac{d}{\sigma}$ and all $\tilde w\in A_{p/\kappa}(\R^d)\cap RH_{\sigma}(\R^d)$. By choosing $\lambda=\lambda_1$ and $\sigma=\sigma_1$ to be the fixed numbers appearing in the statement of the Corollary, $[b,T]$ is bounded on $\mathcal L^{p,\lambda_1}(\tilde w)$ for all $p\in(\kappa,\infty)$ and all $\tilde w\in A_{p/\kappa}(\R^d)\cap RH_{\sigma_1}(\R^d)$. In particular, $[b,T]$ is bounded on $\mathcal L^{p,\lambda_1}(\tilde w)$ for all $p\in(p_{-},p_{+})$ and all $\tilde w\in A_{p/p_{-}}(\R^d)\cap RH_{(p_{+}/p)'}(\R^d)$ where we recall that $p_{-}=\max\bigg\{\kappa,p_{+}\bigg(1-\frac{1}{\sigma_1}\bigg)\bigg\}$ and  $1\leq p_{-}<p_{+}<\infty$.
By assumption, $[b,T]$ is compact on $\mathcal L^{p_1,\lambda_1}(w_1)$ for some
$p_1\in[p_{-},p_{+}]$ and some $w_1\in A_{p_1/p_{-}}(\R^d)\cap RH_{(p_{+}/p_1)'}(\R^d)$. Thus the assumptions of Theorem \ref{thm:main} hold for the operator $[b,T]$ in place of $T$. Hence, the conclusion of Theorem \ref{thm:main} gives the compactness of $[b,T]$ on $\mathcal L^{p,\lambda_1}(w)$ for all $p\in(p_{-},p_{+})$ and all $w\in A_{p/p_{-}}(\R^d)\cap RH_{(p_{+}/p)'}(\R^d)$.
\end{proof}

\begin{corollary}\label{cor:commuBasic 2}
Let $1\leq p_{-}<p_{+}<\infty$, and $1<\sigma_1<\infty$, and $0<\lambda_1<d-\frac{d}{\sigma_1}$, and 
\begin{equation*}
   q_{+}=p_{+}\bigg(1-\frac{1}{\sigma_1}-\frac{\lambda_1}{d}\bigg),\qquad q_{-}=\max\bigg\{q_{+}\bigg(1-\frac{1}{\sigma_1}\bigg),p_{-}\bigg\}.
\end{equation*}
Suppose moreover that:
\begin{enumerate}
  \item $T$ is a linear operator defined and bounded on $L^{\tilde p_1}(\tilde w)$ for some $\tilde p_1\in(p_{-},p_{+})$ and for all $\tilde w\in A_{\tilde p_1/p_{-}}(\R^d)\cap RH_{(p_{+}/\tilde p_1)'}(\R^d)$,
  \item the commutator $[b,T]$ is compact on $\mathcal L^{p_1,\lambda_1}(w_1)$ for some 
  \begin{equation*}
  b\in\BMO(\R^d),
  \end{equation*}
  some $p_1\in[q_{-},q_{+}]$ and some $w_1\in A_{p_1/q_{-}}(\R^d)\cap RH_{(q_{+}/p_1)'}(\R^d)$.
\end{enumerate}
Then $[b,T]$ is compact on $\mathcal L^{p,\lambda_1}(w)$ for the same $b$, for all $p\in(q_{-},q_{+})$ and all $w\in A_{p/q_{-}}(\R^d)\cap RH_{(q_{+}/p)'}(\R^d)$.
\end{corollary}

\begin{proof}
We verify the assumptions of Theorem \ref{thm:main} for the fixed numbers $\lambda_1,p_1,\sigma_1,p_{-},p_{+},d$ and the weight $w_1$ appearing in the statement of the Corollary, and the operator $[b,T]$ in place of $T$. By Theorem \ref{coro:restricted}, $[b,T]$ is bounded on  $L^{\tilde p_1}(\tilde w)$ for the same $\tilde p_1\in(p_{-},p_{+})$ and for all $\tilde w\in A_{\tilde p_1/p_{-}}(\R^d)\cap RH_{(p_{+}/\tilde p_1)'}(\R^d)$. Then, by Theorem \ref{thm:extr. bdd. 2} with $[b,T]$ in place of $T$, for all $p\in(p_{-},p_{+})$ and $\sigma>(p_{+}/p)'$, $[b,T]$ is bounded on $\mathcal L^{p,\lambda}(\tilde w)$, for all  $0<\lambda<d\bigg(1-\frac{p}{p_{+}}-\frac{1}{\sigma}\bigg)$ and all $\tilde w\in A_{p/p_{-}}(\R^d)\cap RH_{\sigma}(\R^d)$. Equivalently, by expressing the range of parameter $p$ in terms of $\lambda$, $[b,T]$ is bounded on $\mathcal L^{p,\lambda}(\tilde w)$, for all
$p\in\bigg(p_{-},p_{+}\bigg(1-\frac{1}{\sigma}-\frac{\lambda}{\sigma}\bigg)\bigg)$ with $0<\lambda<d-\frac{d}{\sigma}$ and all $\tilde w\in A_{p/p_{-}}(\R^d)\cap RH_{\sigma}(\R^d)$. By choosing $\lambda=\lambda_1$ and $\sigma=\sigma_1$ to be the fixed numbers appearing in the statement of the Corollary, $[b,T]$ is bounded on $\mathcal L^{p,\lambda_1}(\tilde w)$ for all $p\in(p_{-},q_{+})$ and all $\tilde w\in A_{p/p_{-}}(\R^d)\cap RH_{\sigma_1}(\R^d)$ where we recall that $q_{+}=p_{+}\bigg(1-\frac{1}{\sigma_1}-\frac{\lambda_1}{d}\bigg)$ and $1\leq p_{-}<q_{+}<\infty$. In particular, $[b,T]$ is bounded on $\mathcal L^{p,\lambda_1}(\tilde w)$ for all $p\in(q_{-},q_{+})$ and all $\tilde w\in A_{p/q_{-}}(\R^d)\cap RH_{(q_{+}/p)'}(\R^d)$ where we recall that 
$q_{-}=\max\bigg\{q_{+}\bigg(1-\frac{1}{\sigma_1}\bigg),p_{-}\bigg\}$ and $1\leq q_{-}<q_{+}<\infty$. By assumption, $[b,T]$ is compact on $\mathcal L^{p_1,\lambda_1}(w_1)$ for some $p_1\in[q_{-},q_{+}]$ and some $w_1\in A_{p_1/q_{-}}(\R^d)\cap RH_{(q_{+}/p_1)'}(\R^d)$. Thus the assumptions of Theorem \ref{thm:main} hold for the operator $[b,T]$ in place of $T$. Hence, the conclusion of Theorem \ref{thm:main} gives the compactness of $[b,T]$ on $\mathcal L^{p,\lambda_1}(w)$ for all $p\in(q_{-},q_{+})$ and all $w\in A_{p/q_{-}}(\R^d)\cap RH_{(q_{+}/p)'}(\R^d)$.
\end{proof}

\section{Commutators of Calder\'on--Zygmund singular integrals}\label{sec:5}

In our application below, we consider Calder\'on--Zygmund singular integral operators which are defined as follows: 

$T$ is a linear operator defined on a suitable class of test functions on $\R^d$, and it has the representation
\begin{equation*}
  Tf(x)=\int_{\R^d}K(x,y)f(y)\ud y,\qquad x\notin\supp f,
\end{equation*}
where the kernel $K$ satisfies the {\em size condition} 
\begin{equation*}
  \abs{K(x,y)}\lesssim\frac{1}{\abs{x-y}^d}
\end{equation*}
and, for some $\delta\in(0,1]$, the {\em smoothness condition}
\begin{equation*}
  \abs{K(x,y)-K(z,y)}+\abs{K(y,x)-K(y,z)}\lesssim \frac{\abs{x-z}^\delta}{\abs{x-y}^{d+\delta}},
\end{equation*}
for all $x,z,y\in\R^d$ such that $\abs{x-y}>\frac12\abs{x-z}$. 
We will need the following classical result of  Coifman--Fefferman \cite{CF}:

\begin{theorem}[\cite{CF}]\label{thm:CF}
Let $T$ be a Calder\'on--Zygmund operator that extends to a bounded operator on $L^2(\R^d)$. Then $T$ extends to a bounded operator on $L^p(w)$ for all $p\in(1,\infty)$ and all $w\in A_p(\R^d)$.
\end{theorem}

The following result of Arai--Nakai \cite{AN}, based on Sawano and Shirai's method in \cite{SS}, provides a concrete condition to verify the assumptions of Corollary \ref{cor:commuBasic 1}:

\begin{theorem}[\cite{AN}]\label{thm:AN}
Let $0<\lambda<d$, $1<p<\infty$ and $T$ be a Calder\'on--Zygmund operator that extends to a bounded operator on $L^2(\R^d)$. Assume also that for all $f\in C_c^\infty(\R^d)$ we have $Tf(x)=\lim_{\eps\to 0}\int_{|x-y|\ge\eps}K(x,y)f(y)dy$ a.e. $x\in\R^d$. If $b\in\CMO(\R^d)$, then $[b,T]$ is compact on the unweighted $\mathcal L^{p,\lambda}(\R^d)$.
\end{theorem}

\begin{remark}
In \cite{AN}, Theorem \ref{thm:AN} is stated and proved in the setting of generalized Morrey space $L^{(p,\varphi)}(\R^d)$, where $\varphi:\R^d\times(0,\infty)\to(0,\infty)$ is a variable growth function. This space coincides with $\mathcal L^{p,\lambda}(\R^d)$ by choosing  $\varphi(x,r)=|B(x,r)|^{\frac{\lambda}{d}-1}$, where  $B=B(x,r)$ is a ball with center $x$ and radius $r$. 
\end{remark}

By combining Corollary \ref{cor:commuBasic 1}, Theorem \ref{thm:CF} and Theorem \ref{thm:AN} we obtain the following result which appears to be new:

\begin{lemma}\label{app2}
Let $1<p_{+}<\infty$, and $1<\sigma<\infty$, and 
\begin{equation*}
  p_{-}=\max\bigg\{1,p_{+}\bigg(1-\frac{1}{\sigma}\bigg)\bigg\},
\end{equation*}
and $0<\lambda<d-\frac{d}{\sigma}$. Suppose moreover that $b\in\CMO(\R^d)$, $T$ is a Calder\'on--Zygmund operator that extends boundedly to $L^2(\R^d)$ and satisfies for all $f\in C_c^\infty(\R^d)$ the condition $Tf(x)=\lim_{\eps\to 0}\int_{|x-y|\ge\eps}K(x,y)f(y)dy$ a.e. $x\in\R^d$. Then the commutator $[b,T]$ is compact on $\mathcal L^{p,\lambda}(w)$ for all $p\in(p_{-},p_{+})$ and all $w\in A_{p/p_{-}}(\R^d)\cap RH_{(p_{+}/p)'}(\R^d)$.
\end{lemma}

\begin{proof}
Let us fix some $\sigma_1\in(1,\infty)$, $\lambda_1\in(0,d-\frac{d}{\sigma_1})$, $p_1\in[p_{-},p_{+}]$ and  $\tilde{p_1}\in(1,\infty)$ for which we verify the assumptions of Corollary \ref{cor:commuBasic 1} with $\kappa=1$. By Theorem \ref{thm:CF}, $T$ extends to a bounded operator on $L^{\tilde p_1}(\tilde w)$ for all $\tilde w\in A_{\tilde p_1}(\R^d)$. By Theorem \ref{thm:AN}, $[b,T]$ is a compact operator on $\mathcal L^{p_1,\lambda_1}(\R^d)=\mathcal L^{p_1,\lambda_1}(w_1)$ with $w_1\equiv 1\in A_{p_1/p_{-}}(\R^d)\cap RH_{(p_{+}/p_1)'}(\R^d)$.
Thus Corollary \ref{cor:commuBasic 1} applies to give the compactness of $[b,T]$ on $\mathcal L^{p,\lambda_1}(w)$ for all $p\in(p_{-},p_{+})$ and all $w\in A_{p/p_{-}}(\R^d)\cap RH_{(p_{+}/p)'}(\R^d)$.
\end{proof}

Now, we check the conditions of Lemma \ref{app2} in such a way that we obtain the following:

\begin{proof}[Proof of Theorem \ref{thm:app2}]
Under the conditions appearing in the assumptions and conclusion of Theorem, we check that we can find parameters $\sigma$ and  $p_{+},p_{-}$ as in the Lemma \ref{app2}. In particular, by choosing
$\sigma\in\bigg(\bigg(\frac{d}{\lambda}\bigg)',(t's)'\bigg]$
$p_{+}\in\bigg[t'p,\frac{\sigma'p}{s}\bigg]$, and $p_{-}=\max\bigg\{1,p_{+}\bigg(1-\frac{1}{\sigma}\bigg)\bigg\}$, Lemma \ref{app2} applies to give the compactness of $[b,T]$ on $\mathcal{L}^{p,\lambda}(w)$ for all $p\in(p_{-},p_{+})$ and all
$w\in A_{p/p_{-}}(\R^d)\cap RH_{(p_{+}/p)'}(\R^d)$. It remains to check that this covers all $w\in A_{s}(\R^d)\cap RH_{t}(\R^d)$ as in the statement of the Theorem. Since $p/p_{-}\ge s$ and $p_{+}/p\ge t'$, the monotonicity of the $A_s(\R^d)$ and $RH_{t}(\R^d)$ classes implies that $[b,T]$ is compact on $\mathcal{L}^{p,\lambda}(w)$ for all $w\in A_{s}(\R^d)\cap RH_{t}(\R^d)$ and all $p,\lambda,s,t$ such that
\begin{equation*}
  p\in(1,\infty),\quad 0<\lambda<d, \quad s\in\bigg[1,\min\bigg\{p,\frac{d}{\lambda}\bigg\}\bigg], \quad t\in\bigg(\bigg(\frac{d}{s\lambda}\bigg)',\infty\bigg).  
\end{equation*}
\end{proof}

\section{Commutators of rough singular integrals}\label{sec:6}

Let $\Omega$ be homogeneous of degree zero in $\R^d$, integrable, and have mean value zero on the unit sphere $S^{d-1}$. Define the singular integral operator $T_{\Omega}$ by
\begin{equation*}
  T_{\Omega}f(x):=\lim_{\eps\to0}\int_{|x-y|>\eps}\frac{\Omega(x-y)}{|x-y|^d}f(y)dy.
\end{equation*}
Duoandikoetxea \cite{Duo:TAMS} and Watson \cite{Watson} considered the following weighted estimates for $T_{\Omega}$:

\begin{theorem}[\cite{Duo:TAMS,Watson}]\label{thm:DW}
Let $r\in(1,\infty)$ and $\Omega\in L^r(S^{d-1})$  be homogeneous of order zero with vanishing mean on $S^{d-1}$. Then $T_\Omega$ extends to a bounded operator on $L^p(w)$ for all $p\in(r',\infty)$ and all $w\in A_{p/r'}(\R^d)$.
\end{theorem}

The unweighted compactness result about the commutator $[b,T_{\Omega}]$ is due to Guo--Hu \cite{GuoHu:16}:

\begin{theorem}[\cite{GuoHu:16}, Theorem 1.8]\label{thm:Chinese}
Let $0<\lambda<d$, $r\in(1,\infty)$ and $\Omega\in L^r(S^{d-1})$ be homogeneous of order zero with vanishing mean on $S^{d-1}$. Let $b\in\CMO(\R^d)$. Then the commutator $[b,T_\Omega]$ is compact on $\mathcal{L}^{p,\lambda}(\R^d)$ for all $p\in(r',\infty)$.
\end{theorem}

A combination of Corollary \ref{cor:commuBasic 1} and Theorem \ref{thm:DW}, together with Theorem \ref{thm:Chinese}, gives the following new result:

\begin{lemma}\label{app3}
Let $1\leq r'<p_{+}<\infty$, and $1<\sigma<\infty$, and 
\begin{equation*}
  p_{-}=\max\bigg\{r',p_{+}\bigg(1-\frac{1}{\sigma}\bigg)\bigg\},
\end{equation*}
and $0<\lambda<d-\frac{d}{\sigma}$, and $r\in(1,\infty)$ and $\Omega\in L^r(S^{d-1})$ be homogeneous of order zero with vanishing mean on $S^{d-1}$. Let $b\in\CMO(\R^d)$. Then the commutator $[b,T_\Omega]$ is compact on $\mathcal{L}^{p,\lambda}(w)$ for all $p\in(p_{-},p_{+})$ and all $w\in A_{p/p_{-}}(\R^d)\cap RH_{(p_{+}/p)'}(\R^d)$. 
\end{lemma}

\begin{proof}
We verify the assumptions of Corollary \ref{cor:commuBasic 1} with $\kappa=r'$, an arbitrary $\sigma_1\in(1,\infty)$, $\lambda_1\in(0,d-\frac{d}{\sigma_1})$, $p_1\in[p_{-},p_{+}]$, $\tilde{p_1}\in(r',\infty)$ and $w_1\equiv 1\in A_{p_1/p_{-}}(\R^d)\cap RH_{(p_{+}/p_1)'}(\R^d)$. Theorem \ref{thm:Chinese} guarantees that $[b,T_\Omega]$ is compact on $\mathcal{L}^{p_1,\lambda_1}(w_1)$ for the exponent $p_1\in[p_{-},p_{+}]$ and weight $w_1\equiv 1\in A_{p_1/p_{-}}(\R^d)\cap RH_{(p_{+}/p_1)'}(\R^d)$. On the other hand, a direct application of Theorem \ref{thm:DW} shows that $T_\Omega$ is bounded on $L^{\tilde p_1}(\tilde w)$ for all $\tilde w\in A_{\tilde p_1/r'}(\R^d)$. Thus Corollary \ref{cor:commuBasic 1} applies to give the compactness of $[b,T_\Omega]$ on $\mathcal{L}^{p,\lambda_1}(w)$ for all $p\in(p_{-},p_{+})$ and all $w\in A_{p/p_{-}}(\R^d)\cap RH_{(p_{+}/p)'}(\R^d)$.
\end{proof}

Now, we check the conditions of Lemma \ref{app3} in such a way that we obtain the following:

\begin{theorem}\label{thm:app3}
Let $r\in(1,\infty)$ and $\Omega\in L^r(S^{d-1})$ be homogeneous of order zero with vanishing mean on $S^{d-1}$. Let $b\in\CMO(\R^d)$. Then the commutator $[b,T_\Omega]$ is compact on $\mathcal{L}^{p,\lambda}(w)$ for all $w\in A_{s}(\R^d)\cap RH_{t}(\R^d)$ and all $p,\lambda,s,t$ such that
\begin{equation*}
  p\in(r',\infty),\quad 0<\lambda<d, \quad s\in\bigg[1,\min\bigg\{\frac{p}{r'},\frac{d}{\lambda}\bigg\}\bigg], \quad t\in\bigg(\bigg(\frac{d}{s\lambda}\bigg)',\infty\bigg).    
\end{equation*}
\end{theorem}

\begin{proof}
Under the conditions appearing in the assumptions and conclusion of Theorem, we check that we can find parameters $\sigma$ and $p_{+},p_{-}$ as in the Lemma \ref{app3}. In particular, by choosing  $\sigma\in\bigg(\bigg(\frac{d}{\lambda}\bigg)',(t's)'\bigg]$, $p_{+}\in\bigg[t'p,\frac{\sigma'p}{s}\bigg]$ and $p_{-}=\max\bigg\{r',p_{+}\bigg(1-\frac{1}{\sigma}\bigg)\bigg\}$, Lemma \ref{app3} applies to give the compactness of $[b,T_\Omega]$ on $\mathcal{L}^{p,\lambda}(w)$ for all $p\in(p_{-},p_{+})$ and all $w\in A_{p/p_{-}}(\R^d)\cap RH_{(p_{+}/p)'}(\R^d)$. It remains to check that this covers all $w\in A_{s}(\R^d)\cap RH_{t}(\R^d)$ as in the statement of the Theorem. Since $p/p_{-}\ge s$ and $p_{+}/p\ge t'$, the monotonicity of the $A_s(\R^d)$ and $RH_{t}(\R^d)$ classes implies that $[b,T_\Omega]$ is compact on $\mathcal{L}^{p,\lambda}(w)$ for all $w\in A_{s}(\R^d)\cap RH_{t}(\R^d)$ and all $p,\lambda,s,t$ such that
\begin{equation*}
  p\in(r',\infty),\quad 0<\lambda<d, \quad s\in\bigg[1,\min\bigg\{\frac{p}{r'},\frac{d}{\lambda}\bigg\}\bigg], \quad t\in\bigg(\bigg(\frac{d}{s\lambda}\bigg)',\infty\bigg).    
\end{equation*}
\end{proof}

\section{Commutators of Bochner--Riesz multipliers}\label{sec:7}

In this section we will apply Theorem \ref{thm:main} to the commutators of Bochner--Riesz multipliers in dimensions $d\ge 2$. Following \cite{HL, LMR2019}, we recall that Bochner--Riesz multiplier is a Fourier multiplier $B^\kappa$ with the symbol $(1-|\xi|^2)_{+}^{\kappa}$, where $\kappa>0$ and $t_{+}=\max(t,0)$. That is, the Bochner--Riesz operator is defined, on the class $\testi(\R^d)$ of Schwartz functions, by
\begin{equation*}
  \widehat {B^\kappa f} (\xi)=(1-|\xi|^2)_{+}^{\kappa}\widehat f(\xi),
\end{equation*}
where $\widehat f$ denotes the Fourier transform of $f$.

The following Bochner--Riesz conjecture is well-known: (See also the works of \cite{CS,Cordoba1979} in two dimensions and \cite{BG,Lee} in the case $d\ge3$.)

\begin{conjecture}[Bochner--Riesz Conjecture]\label{j:BR} 
For $0<\kappa<\frac{d-1}{2}$, we have $B^\kappa: L ^{p}(\R ^{d})\mapsto L ^{p}(\R^{d})$ if 
\begin{equation*}
  p\in\bigg(\frac{2d}{d+1+2\kappa},\frac{2d}{d-1-2\kappa}\bigg). 
\end{equation*}
\end{conjecture}

In \cite{LMR2019}, an equivalent form of the Bochner--Riesz Conjecture \ref{j:BR} is stated as follows: (See also \cite{Carbery, Cordoba1977, Cordoba1979} and \cite[Chapter 8.5]{D2001} for the connection between the Bochner--Riesz and the $S_{\tau}$ Fourier multiplier)
\begin{conjecture}\label{conj.}
Let $\mathbf 1_{[-1/4,1/4]}\leq\chi\leq\mathbf 1_{[-1/2,1/2]}$ be a Schwartz function and denote by $S_{\tau}$ the Fourier multiplier with symbol  $\chi((|\xi|-1)/\tau)$. If $\frac{2d}{d+1}<p<\frac{2d}{d-1}$, then 
\begin{equation}\label{eq:conjecture}
  \|S_{\tau}\|_{L^p(\R^d)\mapsto L^p(\R^d)}\leq C_{\epsilon}\tau^{-\epsilon},
\end{equation}
where $0<\tau<1$ and $C_{\epsilon}$ is a constant that depends on $0<\epsilon<1$.
\end{conjecture}

The following weighted estimates for $B^{\kappa}$ were obtained in \cite{HL}:

\begin{theorem}[\cite{HL}, Corollary 10.5]\label{thm:Bochner-Riesz limited range estimate}
If $d=2$, $0<\kappa<\frac{1}{2}$ and 
\begin{equation*}
  p\in\bigg(\frac{4}{1+6\kappa},\frac{4}{1-2\kappa}\bigg),
\end{equation*}
then $B^{\kappa}$ is bounded on $L^p(w)$ for all 
\begin{equation*}
  w\in A_{\frac{p(1+6\kappa)}{4}}(\R^2)\cap RH_{\big(\frac{4}{p(1-2\kappa)}\big)'}(\R^2).
\end{equation*}
\newline Moreover, if $d\ge 3$, $0<\kappa<\frac{d-1}{2}$, $1<p_0<2$ is such that the estimate \eqref{eq:conjecture} of Conjecture \ref{conj.} holds, and 
\begin{equation*}
  p\in\bigg(\frac{p_0(d-1)}{d-1+2\kappa(p_0-1)},\frac{p_0(d-1)}{d-1-2\kappa}\bigg),
\end{equation*}
then $B^{\kappa}$ is bounded on $L^p(w)$ for all
\begin{equation*}
  w\in A_{\frac{p(d-1+2\kappa(p_0-1))}{p_0(d-1)}}(\R^d)\cap RH_{\big(\frac{p_0(d-1)}{p(d-1-2\kappa)}\big)'}(\R^d).
\end{equation*} 
\end{theorem}

In \cite{BCH}, Bu--Chen--Hu proved the following unweighted compactness result:

\begin{theorem}[\cite{BCH}, Theorems 1.4 and 1.5]\label{thm:cmpt. Bochner-Riesz limited range estimate}
If $d=2$, $0<\kappa<\frac{1}{2}$,
\begin{equation*}
  p\in\bigg(\frac{4}{3+2\kappa},\frac{4}{1-2\kappa}\bigg),
\end{equation*}
and 
\begin{equation*}
\begin{split}
  &\lambda\in(0,2\kappa\theta_p/(2\kappa\theta_p+1-2\kappa))\qquad\text{with}  \\ \theta_p&=\frac{1}{1+2\kappa}\min\{4/p-(1-2\kappa),3+2\kappa-4/p\},
\end{split}
\end{equation*}
then for $b\in\CMO(\R^2)$, the commutator $[b,B^{\kappa}]$ is compact on $\mathcal L^{p,\lambda}(\R^2)$.
\newline Moreover, if $d\ge3$,
$\frac{d-1}{2d+2}<\kappa<\frac{d-1}{2}$
\begin{equation*}
  p\in\bigg(\frac{2d}{d+1+2\kappa},\frac{2d}{d-1-2\kappa}\bigg),    
\end{equation*}
and 
\begin{equation*}
\begin{split}
  &\lambda\in(0,2\kappa\theta_p/(2\kappa\theta_p+d-1-2\kappa))\qquad\text{with}  \\
  \theta_p&=\frac{1}{1+2\kappa}\min\{2d/p-(d-1-2\kappa),d+1+2\kappa-2d/p\},
\end{split}
\end{equation*}
then for $b\in\CMO(\R^d)$, the commutator $[b,B^{\kappa}]$ is compact on $\mathcal L^{p,\lambda}(\R^d)$.
\end{theorem}

By combining Corollary \ref{cor:commuBasic 2} and Theorem \ref{thm:Bochner-Riesz limited range estimate}, together with Theorem \ref{thm:cmpt. Bochner-Riesz limited range estimate}, we obtain the following new weighted of compactness result:

\begin{lemma}\label{app4}
If $d=2$, $0<\kappa<\frac{1}{2}$, $1<\sigma<\infty$, $0<\lambda<d-\frac{d}{\sigma}$, and 
\begin{equation*}
  q_{+}=\frac{4}{1-2\kappa}\bigg(1-\frac{1}{\sigma}-\frac{\lambda}{d}\bigg),\qquad q_{-}=\max\bigg\{q_{+}\bigg(1-\frac{1}{\sigma}\bigg),\frac{4}{1+6\kappa}\bigg\},
\end{equation*}
then for $b\in\CMO(\R^2)$, the commutator $[b,B^\kappa]$ is compact on $\mathcal L^{p,\lambda}(w)$ for all $p\in(q_{-},q_{+})$ and all
\begin{equation*}
  w\in A_{\frac{p}{q_{-}}}(\R^2)\cap RH_{\big(\frac{q_{+}}{p}\big)'}(\R^2).
\end{equation*}
\newline Moreover, if $d\ge 3$, $\frac{d-1}{2d+2}<\kappa<\frac{d-1}{2}$, $1<\sigma<\infty$, $0<\lambda<d-\frac{d}{\sigma}$, $1<p_0<2$ is such that the estimate \eqref{eq:conjecture} of Conjecture \ref{conj.} holds and 
\begin{equation*}
\begin{split}
  q_{+}&=\frac{p_0(d-1)}{d-1-2\kappa}\bigg(1-\frac{1}{\sigma}-\frac{\lambda}{d}\bigg),  \\ q_{-}&=\max\bigg\{q_{+}\bigg(1-\frac{1}{\sigma}\bigg),\frac{p_0(d-1)}{d-1+2\kappa(p_0-1)}\bigg\},
\end{split}
\end{equation*}
then for $b\in\CMO(\R^d)$, the commutator $[b,B^\kappa]$ is compact on $\mathcal L^{p,\lambda}(w)$ for all $p\in(q_{-},q_{+})$ and all
\begin{equation*}
  w\in A_{\frac{p}{q_{-}}}(\R^d)\cap RH_{\big(\frac{q_{+}}{p}\big)'}(\R^d).
\end{equation*}
\end{lemma}

\begin{proof}
Let $d\ge 3$, $\frac{d-1}{2d+2}<\kappa<\frac{d-1}{2}$ and $p_0$ be as in the assumptions. We verify the assumptions of Corollary \ref{cor:commuBasic 2} for the fixed exponents $\sigma_1\in(1,\infty)$, $\lambda_1\in(0,d-\frac{d}{\sigma_1})$, $p_1\in[q_{-},q_{+}]$
and $\tilde p_1\in(p_{-},p_{+})$ with
$p_{-}=\frac{p_0(d-1)}{d-1+2\kappa(p_0-1)}$,  $p_{+}=\frac{p_0(d-1)}{d-1-2\kappa}$ and $1<p_{-}<p_{+}<\infty$. By Theorem \ref{thm:Bochner-Riesz limited range estimate}, $B^\kappa$ is a bounded operator on $L^{\tilde p_1}(\tilde w)$ for all
\begin{equation*}
  \tilde w\in A_{\frac{\tilde{p_1}}{p_{-}}}(\R^d)\cap RH_{\big(\frac{p_{+}}{\tilde p_1}\big)'}(\R^d).
\end{equation*}
By Theorem \ref{thm:cmpt. Bochner-Riesz limited range estimate}, $[b,B^\kappa]$ is a compact operator on $\mathcal L^{p_1,\lambda_1}(\R^d)=\mathcal L^{p_1,\lambda_1}(w_1)$ with
\begin{equation*}
  w_1\equiv 1\in A_{\frac{p_1}{q_{-}}}(\R^d)\cap RH_{\big(\frac{q_{+}}{p_1}\big)'}(\R^d).
\end{equation*}
Thus Corollary \ref{cor:commuBasic 2} applies to give the compactness of $[b,B^\kappa]$ on $\mathcal L^{p,\lambda_1}(w)$ for all $p\in(q_{-},q_{+})$ and all
\begin{equation*}
  w\in A_{\frac{p}{q_{-}}}(\R^d)\cap RH_{\big(\frac{q_{+}}{p}\big)'}(\R^d).
\end{equation*}
The case $d=2$ follows in a similar way.
\end{proof}

Now, we check the conditions of Lemma \ref{app4} in such a way that we obtain the following:

\begin{theorem}\label{thm:app4}
\begin{enumerate}
  \item\label{thm:case1} Let $d=2$, $0<\kappa<\frac{1}{2}$ and denote $p_{-}=\frac{4}{1+6\kappa}$. Then for $b\in\CMO(\R^2)$, the commutator $[b,B^\kappa]$ is compact on $\mathcal L^{p,\lambda}(w)$ for all $w\in A_{s}(\R^2)\cap RH_{t}(\R^2)$ and all $p,\lambda,s,t$ such that 
  \begin{equation*}
  p\in(p_{-},\infty),\quad 0<\lambda<d, \quad s\in\bigg[1,\min\bigg\{\frac{p}{p_{-}},\frac{d}{\lambda}\bigg\}\bigg], \quad t\in\bigg(\bigg(\frac{d}{s\lambda}\bigg)',\infty\bigg).
  \end{equation*}
  \item\label{thm:case2} Let $d\ge 3$, $\frac{d-1}{2d+2}<\kappa<\frac{d-1}{2}$, $1<p_0<2$ be such that the estimate \eqref{eq:conjecture} of Conjecture \ref{conj.} holds and denote $p_{-}=\frac{p_0(d-1)}{d-1+2\kappa(p_0-1)}$. 
  Then for $b\in\CMO(\R^d)$, the commutator $[b,B^\kappa]$ is compact on $\mathcal L^{p,\lambda}(w)$ for all $w\in A_{u}(\R^d)\cap RH_{v}(\R^d)$ and all $p,\lambda,u,v$ such that 
  \begin{equation*}
  p\in(p_{-},\infty), \quad 0<\lambda<d, \quad
  u\in\bigg[1,\min\bigg\{\frac{p}{p_{-}},\frac{d}{\lambda}\bigg\}\bigg], \quad v\in\bigg(\bigg(\frac{d}{u\lambda}\bigg)',\infty\bigg).     
  \end{equation*}
\end{enumerate}
\end{theorem}

\begin{proof}
We prove the Theorem in the case \eqref{thm:case2}. Under these assumptions and the conditions appearing in the conclusion of Theorem, we check that we can find parameters $\sigma$ and $q_{+},q_{-}$ as in the Lemma \ref{app4}. In particular, by choosing $\sigma\in\bigg(\bigg(\frac{d}{\lambda}\bigg)',(v'u)'\bigg]$, $q_{+}\in\bigg[v'p,\frac{\sigma'p}{u}\bigg]$ and $q_{-}=\max\bigg\{q_{+}\bigg(1-\frac{1}{\sigma}\bigg),p_{-}\bigg\}$, Lemma \ref{app4} applies to give the compactness of $[b,B^\kappa]$ on $\mathcal L^{p,\lambda}(w)$ for all $p\in(q_{-},q_{+})$ and all
\begin{equation*}
  w\in A_{\frac{p}{q_{-}}}(\R^d)\cap RH_{\big(\frac{q_{+}}{p}\big)'}(\R^d).
\end{equation*}
It remains to check that this covers all $w\in A_{u}(\R^d)\cap RH_{v}(\R^d)$ as in the statement of the Theorem. Since $p/q_{-}\ge u$ and $q_{+}/p\ge v'$, the monotonicity of the $A_u(\R^d)$ and $RH_{v}(\R^d)$ classes implies that $[b,B^\kappa]$ is compact on $\mathcal{L}^{p,\lambda}(w)$ for all $w\in A_{u}(\R^d)\cap RH_{v}(\R^d)$ and all $p,\lambda,u,v$ such that 
\begin{equation*}
  p\in(p_{-},\infty), \quad 0<\lambda<d, \quad
  u\in\bigg[1,\min\bigg\{\frac{p}{p_{-}},\frac{d}{\lambda}\bigg\}\bigg], \quad v\in\bigg(\bigg(\frac{d}{u\lambda}\bigg)',\infty\bigg).     
\end{equation*}
\newline The case \eqref{thm:case1} of the Theorem follows in a similar way.
\end{proof}

\subsection*{Acknowledgements} The author would like to express his gratitude to his doctoral supervisor Prof. Tuomas Hyt{\"o}nen for his kind support. The author was supported by the Academy of Finland through the grant No. 314829 and by the Foundation for Education and European Culture (Founders Nicos and Lydia Tricha).

\subsection*{Declarations of interest:} none.


\begin{thebibliography}{10}

\normalsize
\baselineskip=17pt

\bibitem{ABKP}
J.~\'{A}lvarez, R.~J. Bagby, D.~S. Kurtz and C.~P\'{e}rez.
\newblock Weighted estimates for commutators of linear operators.
\newblock {\em Studia Math.}, 104(2):195--209, 1993.

\bibitem{AN}
R.~Arai and E.~Nakai.
\newblock Compact commutators of Calder\'on--Zygmund and generalized fractional integral operators with a function in generalized Campanato spaces on generalized Morrey spaces.
\newblock {\em Tokyo J. Math.}, 42(2):471--496, 2019.

\bibitem{AM}
P.~Auscher and J.~Martell.
\newblock Weighted norm inequalities, off-diagonal estimates and elliptic operators. Part I: General operator theory and weights.
\newblock {\em Adv. Math.}, 212(1):225--276, 2007.

\bibitem{BMMST}
A.~B\'enyi, J.~M. Martell, K.~Moen, E.~Stachura and R.~H. Torres.
\newblock Boundedness results for commutators with {BMO} functions via weighted
estimates: a comprehensive approach.
\newblock {\em Math. Ann.}, 376(1-2):61--102, 2020.

\bibitem{BG}
J.~Bourgain and L.~Guth.
\newblock Bounds on oscillatory integral operators based on multilinear estimates.
\newblock {\em Geom. Funct. Anal.}, 21(6):1239--1295, 2011.

\bibitem{BCH}
R.~Bu, J.~Chen and G.~Hu.
\newblock Compactness for the commutator of Bochner--Riesz operator.
\newblock {\em Acta Math. Sci. Ser. B}, 37(5):1373--1384, 2017.

\bibitem{Calderon}
A.~P. Calder\'on.
\newblock Intermediate spaces and interpolation, the complex method. 
\newblock {\em Studia Math.}, 24:113--190, 1964.

\bibitem{COY}
M.~Cao, A.~Olivo and K.~Yabuta.
\newblock Extrapolation for multilinear compact operators and applications.
\newblock {\em Preprint}, 2020,
\newblock arXiv:2011.13191v3.

\bibitem{Carbery}
A.~Carbery.
\newblock A weighted inequality for the maximal Bochner-–Riesz operator on $\R^2$.
\newblock {\em Trans. Amer. Math. Soc.}, 287(2):673--680, 1985.

\bibitem{CS}
L.~Carleson and P.~Sj\"olin.
\newblock Oscillatory integrals and a multiplier problem for the disc.
\newblock {\em Studia Math.}, 44:287--299 (errata insert), 1972.
\newblock Collection of articles honoring the completion by Antoni Zygmund
of 50 years of scientific activity, III.

\bibitem{CF}
R.~R. Coifman and C.~Fefferman.
\newblock Weighted norm inequalities for maximal functions and singular integrals.
\newblock {\em Studia Math.}, 51:241--250, 1974.

\bibitem{Cordoba1977}
A.~Cordoba.
\newblock The Kakeya maximal function and the spherical summation multipliers. 
\newblock {\em Am. J. Math.}, 99(1):1--22, 1977.

\bibitem{Cordoba1979}
A.~C\'ordoba.
\newblock A note on {B}ochner--{R}iesz operators.
\newblock {\em Duke Math. J.}, 46(3):505--511, 1979.
       
\bibitem{CUMP:book}
D.~V. Cruz-Uribe, J.~M. Martell and C.~P\'erez.
\newblock {\em Weights, extrapolation and the theory of {R}ubio de {F}rancia},
volume 215 of {\em Operator Theory: Advances and Applications}.
\newblock Birkh\"{a}user/Springer Basel AG, Basel, 2011.

\bibitem{CwKa}
M.~Cwikel and N.~J. Kalton.
\newblock Interpolation of compact operators by the methods of {C}alder\'on
and {G}ustavsson-{P}eetre.
\newblock {\em Proc. Edinburgh Math. Soc.}, 38(2):261--276, 1995.

\bibitem{CR}
M.~Cwikel and R.~Rochberg.
\newblock Lecture notes on complex interpolation of compactness.
\newblock {\em Preprint}, 2014,
\newblock arXiv:1411.0171.

\bibitem{Duo:TAMS}
J.~Duoandikoetxea.
\newblock Weighted norm inequalities for homogeneous singular integrals.
\newblock {\em Trans. Amer. Math. Soc.}, 336(2):869--880, 1993.

\bibitem{D2001}
J.~Duoandikoetxea.
\newblock {\em Fourier Analysis}, volume 29 of {\em Graduate Studies in Mathematics}. 
\newblock American Mathematical Society, Providence, RI, 2001. 
\newblock Translated and revised from the 1995 Spanish original by D. Cruz-Uribe.

\bibitem{DR2018}
J.~Duoandikoetxea and M.~Rosenthal.
\newblock Extension and Boundedness of Operators on Morrey spaces from Extrapolation Techniques and Embeddings.
\newblock {\em J. Geom. Anal.}, 28:3081--3108, 2018.

\bibitem{DR}
J.~Duoandikoetxea and M.~Rosenthal. 
\newblock Boundedness of Operators on Certain Weighted Morrey spaces Beyond the Muckenhoupt Range.
\newblock {\em Pot. Anal.}, 53:1255--1268, 2020.

\bibitem{DR2020}
J.~Duoandikoetxea and M.~Rosenthal. 
\newblock Boundedness properties in a family of weighted Morrey spaces with emphasis on power weights.
\newblock {\em J. Funct. Anal.}, 279(8):108687, 2020.

\bibitem{Gehring}
F.~W. Gehring.
\newblock The $L^p$-integrability of the partial derivatives of a quasiconformal mapping.  
\newblock {\em Acta Math.}, 130:265--277, 1973.

\bibitem{GuoHu:16}
X.~Guo and G.~Hu.
\newblock On the commutators of singular integral operators with rough convolution kernels.
\newblock {\em Canad. J. Math.}, 68(4):816--840, 2016.

\bibitem{HNS}
D.~I. Hakim, S.~Nakamura and Y.~Sawano.
\newblock Weighted Morrey spaces-complex interpolation and the boundedness of the Hardy-Littlewood maximal operator.
\newblock {\em Res. Inst. Math. Sci. (RIMS) Kôkyûroku Bessatsu}, B65:109--140, 2017.

\bibitem{HNVW1}
T.~Hyt\"{o}nen, J.~van Neerven, M.~Veraar and L.~Weis.
\newblock {\em Analysis in {B}anach spaces. {V}ol. {I}. {M}artingales and
{L}ittlewood--{P}aley theory}, volume~63 of {\em Ergebnisse der Mathematik und
ihrer Grenzgebiete. 3. Folge.}
\newblock Springer, Cham, 2016.

\bibitem{HL}
T.~Hyt{\"o}nen and S.~Lappas.
\newblock Extrapolation of compactness on weighted spaces. 
\newblock {\em Preprint}, 2020,
\newblock arXiv:2003.01606v3.

\bibitem{HL2020}
T.~Hyt\"onen and S.~Lappas.
\newblock Extrapolation of compactness on weighted spaces: Bilinear operators.
\newblock {\em Preprint}, 2020,
\newblock arXiv:2012.10407v2, to appear in Indag. Math.

\bibitem{KS}
Y.~Komori and S.~Shirai. 
\newblock Weighted Morrey spaces and a singular integral operator.
\newblock {\em Math. Nachr.}, 282(2):219--231, 2009.

\bibitem{LMR2019}
M.~Lacey, D.~Mena and M.~C. Reguera.
\newblock Sparse bounds for Bochner--Riesz Multipliers.
\newblock {\em J. Fourier Anal. and Appl.}, 25:523--537, 2019.

\bibitem{Lee}
S.~Lee.
\newblock Improved bounds for Bochner--Riesz and maximal Bochner--Riesz operators.
\newblock {\em Duke Math. J.}, 122(1):205--232, 2004.

\bibitem{Morrey}
C.~B. Morrey.
\newblock On the solutions of quasi‐linear elliptic partial differential equations.
\newblock {\em Trans. Amer. Math. Soc.}, 43(1):126--166, 1938.

\bibitem{Muckenhoupt:Ap}
B.~Muckenhoupt.
\newblock Weighted norm inequalities for the {H}ardy maximal function.
\newblock {\em Trans. Amer. Math. Soc.}, 165:207--226, 1972.

\bibitem{Rubio:factorAp}
J.~L. Rubio de Francia.
\newblock Factorization theory and {$A_{p}$} weights.
\newblock {\em Amer. J. Math.}, 106(3):533--547, 1984.

\bibitem{Samko}
N.~Samko.
\newblock Weighted Hardy and singular operators in Morrey spaces.
\newblock {\em J. Math. Anal. Appl.}, 350:56--72, 2009.

\bibitem{SS}
Y.~Sawano and S.~Shirai.
\newblock Compact commutators on Morrey spaces with non-doubling measures.
\newblock {\em Georgian Math. J.}, 15(2):353--376, 2008.

\bibitem{ST}
Y.~Sawano and H.~Tanaka.
\newblock The Fatou property of block spaces.
\newblock {J. Math. Sci. Univ. Tokyo}, 22(3):663--683, 2015.

\bibitem{Watson}
D.~K. Watson.
\newblock Weighted estimates for singular integrals via {F}ourier transform
estimates.
\newblock {\em Duke Math. J.}, 60(2):389--399, 1990.

\end{thebibliography}
\end{document}